\documentclass[11pt,a4paper]{amsart}
\usepackage[T1]{fontenc}
\usepackage[utf8]{inputenc}
\usepackage{lmodern}
\usepackage{amsmath,amssymb,amsthm}
\usepackage[margin=2cm]{geometry}
\usepackage{graphicx}
\usepackage{microtype}
\usepackage{biblatex}
\usepackage{hyperref}
\usepackage{xcolor}
\hypersetup{colorlinks=false,pdfborder={0 0 0},
  pdftitle={Closed geodesics on compact semi-Riemannian manifolds},
  pdfauthor={Souheib Allout and Salah Chaib}}
\newtheorem{theorem}{Theorem}[section]
\newtheorem{lemma}[theorem]{Lemma}
\newtheorem{proposition}[theorem]{Proposition}
\newtheorem{corollary}[theorem]{Corollary}
\newtheorem{problem}[theorem]{Problem}
\theoremstyle{remark}
\newtheorem{remark}[theorem]{Remark}
\newtheorem{example}[theorem]{Example}
\numberwithin{equation}{section}
\newcommand{\R}{\mathbb R}
\newcommand{\Z}{\mathbb Z}
\newcommand{\Q}{\mathbb Q}
\newcommand{\T}{\mathbb T}
\newcommand{\g}{\mathfrak g}
\DeclareMathOperator{\Sol}{Sol}
\DeclareMathOperator{\Euc}{Euc}
\DeclareMathOperator{\Heis}{Heis}

\DeclareMathOperator{\ad}{ad}

\DeclareMathOperator{\diag}{diag}
\DeclareMathOperator{\Span}{span}
\DeclareMathOperator{\tr}{tr}

\title[Closed geodesics on compact manifolds]{Closed geodesics on compact semi-Riemannian manifolds}
\author[S. Allout]{Souheib Allout}
\address{Souheib Allout, Faculty of Mathematics, USTHB university, Algiers, Algeria}
\email{souheib.allout.math@gmail.com}
\author[S. Chaib]{Salah Chaib}
\address{Salah Chaib, Centro de Matem\'atica, Universidade do Minho, Campus de Gualtar, 4710-057 Braga, Portugal}
\email{salah.chaib@cmat.uminho.pt}
\date{October 12, 2026}

\begin{document}
\begin{abstract}
We study the existence of closed geodesics on compact quotients of Lie groups endowed
with left-invariant pseudo-Riemannian metrics. 
We construct a compact, geodesically complete pseudo-Riemannian manifold that admits no closed geodesics. 
We also establish the existence of closed geodesics in several settings, with particular focus on the Lorentzian case.
\end{abstract}
\maketitle

\section{Introduction}
Closed geodesics on compact Riemannian manifolds form a classical subject in differential geometry.
A foundational theorem of Lyusternik and Fet \cite{LyusternikFet1951} asserts that every compact Riemannian manifold admits a (nonconstant) closed geodesic. 
Later, Gromoll and Meyer \cite{GromollMeyer1969} related the multiplicity of closed geodesics to the topology of the free loop space, proving in particular that sufficiently rich loop space topology forces infinitely many geometrically distinct closed geodesics.

The pseudo-Riemannian situation is substantially more delicate. 
The energy functional on the free loop space is no longer bounded from below, and the standard Riemannian variational methods do not apply.
In Lorentzian geometry, several existence and nonexistence results are known.
Tipler established an early existence theorem for closed timelike geodesics \cite{Tipler1979} and later generalized by Galloway \cite{Galloway1984}.
Galloway also constructed an incomplete Lorentzian metric on $\T^3$ without closed non-spacelike geodesics \cite{galloway1986compact}.
S\'anchez \cite{Sanchez2005static} showed that every compact static spacetime admits a closed timelike geodesic.
This result was generalized in \cite{flores2011periodic} to compact Lorentzian manifolds admitting a Killing vector field that is timelike somewhere. 
In dimension two, Suhr \cite{suhr2013closed} proved that every compact Lorentzian surface admits at least two geometrically distinct closed geodesics and that this bound is optimal.
Nevertheless, the existence of closed geodesics on compact pseudo-Riemannian manifolds remains poorly understood.

Recently, the authors of \cite{AlloutBelkacemZeghib2024closed} constructed compact Lorentzian manifolds admitting no closed geodesics. 
These examples, however, are geodesically incomplete. We aim to show that geodesic completeness does not restore the existence theorem in higher signature.

In this paper, we focus on compact quotients $\Gamma\backslash G$, where $G$ is a Lie group endowed with a complete left-invariant pseudo-Riemannian metric and $\Gamma$ is a cocompact lattice.
Our main result is the following.
\begin{theorem}\label{thm:main}
There is a Lie group $G$ equipped with a geodesically complete left-invariant pseudo-Riemannian metric $m$ of signature $(2,4)$ and a cocompact lattice $\Gamma\subset G$ such that $(\Gamma\backslash G,m)$ has no closed geodesic.
\end{theorem}
The group $G$ and the lattice construction follow the approach of \cite{AlloutBelkacemZeghib2024closed}; the essential improvement is the choice of a complete metric on the Sol factor.

We also show that every compact quotient of a Lie group whose Lie algebra has a nontrivial center admits a closed geodesic.
In Lorentzian signature, we do not know whether geodesic completeness forces the existence of closed geodesics.
\begin{problem}\label{prob:lorentz}
Does every compact geodesically complete Lorentz manifold admit a closed geodesic?
\end{problem}
Motivated by this question, we provide examples that have only null or spacelike closed geodesics (see Subsec. \ref{subsec:rem_ex}).

\subsection{Existence of closed geodesics in natural classes}
\subsubsection{Constant sectional curvature}
Carri\`ere \cite{Carriere1989} proved that compact flat Lorentz manifolds are
geodesically complete. Klingler \cite{Klingler1996} extended this result to all compact Lorentz manifolds of
constant sectional curvature. Consequently, their universal covers are the corresponding complete
simply connected Lorentz space forms. The closed-geodesic problem can therefore be expressed in
terms of the deck group $\Gamma$.

\emph{Zero curvature.} The existence of closed
geodesics on compact flat Lorentz manifolds was shown by
Guediri \cite{Guediri2002}.
Our example in Theorem~\ref{thm:null-intro} exhibits a compact flat Lorentz
three-manifold whose closed geodesics are all lightlike.

\emph{Negative curvature.} The finiteness of levels theorem for compact anti-de Sitter manifolds, established by Zeghib \cite{Zeghib1998AdS} and Kulkarni-Raymond \cite{KulkarniRaymond1985}, implies that every timelike geodesic in a compact $\mathrm{AdS}$-manifold is closed.

\emph{Positive curvature.} The Calabi--Markus phenomenon \cite{CalabiMarkus1962} implies that the de Sitter space has no compact quotient.

\subsubsection{The Brinkmann case}
A Brinkmann manifold is a Lorentz manifold carrying a nonzero parallel
lightlike vector field. 
Mehidi and Zeghib \cite{MehidiZeghib2026brinkmann} prove geodesic completeness in the compact
case. 
They further conjecture that the flow generated by the parallel vector field is equicontinuous, i.e., it has compact closure in the isometry group. 
Whenever this property holds, Proposition~\ref{prop:compact-action} gives a closed geodesic. 
Important cases are

\emph{Cahen--Wallach spaces.} These are simply connected,
indecomposable Lorentzian symmetric spaces of nonconstant sectional
curvature. Their compact quotients are Brinkmann (after a possible
double cover). The parallel flow is periodic \cite[Proposition~8.2]{KathOlbrich2019}. Its orbits are closed null geodesics.

\emph{Locally homogeneous plane waves.} A plane wave is a Brinkmann manifold whose leaves tangent to $V^\perp$ are flat and whose curvature tensor satisfies $\nabla_X R=0$ for every $X\in V^\perp$, where $V$ is the parallel lightlike vector field.
Hanounah, Kath, Mehidi and Zeghib \cite[Theorem~1.11]{HanounahKathMehidiZeghib2025} prove that the parallel flow of every compact locally homogeneous plane wave is equicontinuous, although it need not be periodic.

\subsubsection{Locally symmetric case}
The completeness of closed locally symmetric Lorentz manifolds was recently proved in \cite{AlloutHanounah2025}. Building on this result, the existence of closed geodesics on such manifolds is shown in \cite{AlloutClosedLocallySymmetric}.

\subsection{Remarkable examples}\label{subsec:rem_ex}

\subsubsection{A complete example whose closed geodesics are all spacelike}
Galloway \cite{galloway1986compact} constructed an incomplete Lorentzian metric on $\mathbb{T}^3$ without non-spacelike closed geodesics.
Guediri \cite{guediri2007new} later proposed complete examples with the same property, obtained as compact quotients of the Heisenberg group endowed with a left-invariant Lorentzian metric.
However, these quotients admit closed timelike geodesics, as we show in Theorem~\ref{thm:heis-intro}. 
We provide a geodesically complete example with only spacelike closed geodesics.

\begin{theorem}\label{thm:sol-intro}
Let $\Sol$ denote the identity component of the isometry group of
the Lorentz plane $\mathbb{R}^{1,1}$.
There exists a complete left-invariant metric on $\Sol$ such that for every cocompact lattice $\Gamma$, the quotient $\Gamma\backslash \Sol$ is without non-spacelike closed geodesics
\end{theorem}

\subsubsection{A complete flat example whose closed geodesics are all null}
\begin{theorem}\label{thm:null-intro}
There exists a flat left-invariant Lorentzian metric on the Heisenberg group $\mathrm{Heis}_3$ and a cocompact lattice $\Gamma\subset\mathrm{Heis}_3$ such that the closed geodesics on $\Gamma\backslash\mathrm{Heis}_3$ are precisely the fibers of its natural circle fibration over $\mathbb{T}^2$.
All these fibers are null.
\end{theorem}

\subsubsection{Heisenberg quotients}

\begin{theorem}\label{thm:heis-intro}
Every compact Lorentz quotient of the form $\Gamma\backslash\Heis$ with non-degenerate center admits both closed timelike and spacelike geodesics.
\end{theorem}

\subsubsection{Only timelike}
We do not know any example whose closed geodesics are all timelike.

\subsection{Further discussion}
On a compact Lorentz manifold, the existence of a timelike Killing vector field implies the existence of a closed timelike geodesic \cite{Sanchez2005static}. 
It also provides a Riemannian metric whose energy is a first integral of the geodesic flow.
In our work in preparation \cite{AlloutChaibRiemannianIntegrals}, we show that, in Lorentzian signature, the existence of such a first integral alone suffices to ensure the existence of a closed timelike geodesic.
This leaves open the following question.
\begin{problem}
Does there exist a compact pseudo-Riemannian manifold whose geodesic
flow preserves the energy of a Riemannian metric but which has no closed geodesics?
\end{problem}

\section{General construction}
\subsection{The Euler--Arnold equation}
Let $G$ be a connected Lie group and $\g$ be its algebra of left
invariant vector fields. We identify $\g$ with $T_eG$ the tangent space at the identity element $e\in G$.
For a smooth curve $\gamma:I\to G$, defined on a connected interval $I\subset\R$, one associates its velocity
curve $X_\gamma:I\to\g$ as follows
$$
X_\gamma(t)=\bigl(dL_{\gamma(t)^{-1}}\bigr)_{\gamma(t)}(\dot\gamma(t))\in\g.
$$
where $L_g$ denotes the left translation by an element $g\in G$.

Endow $G$ by a left-invariant semi-Riemannian metric $q$, that is, $q$ is obtained by left translating
an inner product, also denoted $q$, on the Lie algebra $\g$. We have
\begin{theorem}[Euler--Arnold Equation, \cite{Arnold-paper}]\label{thm:ea}
Let $V_q:\g\to\g$ be the $2$-homogeneous vector field, called the Euler--Arnold vector field, on $\g$ defined by
\begin{equation}\label{eq:ea}
V_q(x)=\ad_x^*x
\end{equation}
where $\ad_x^*$ denotes the $q$-adjoint of $\ad_x$. Then a smooth curve $\gamma:I\to G$ is a geodesic if and only if $X_\gamma$ is an integral curve for $V_q$.
\end{theorem}
\begin{remark}\label{rem:complete}
Observe that a left-invariant metric $q$ on $G$ is geodesically complete if and only if the Euler--Arnold vector field $V_q$ is complete.
\end{remark}

\subsection{Geodesics on compact quotients}

For a smooth curve $X:I\to\g$ defined on a connected interval $I$ containing $0$, let $\gamma:I\to G$ be the solution of
$$
\dot\gamma(t)=(dL_{\gamma(t)})_eX(t),\qquad \gamma(0)=e.
$$
$\gamma$ is the \emph{developing curve} of $X$ starting at $\gamma(0)=e$.
Assume $X:\R\to\g$ is $T$-periodic. Then
$$
\gamma(t+T)=\gamma(T)\gamma(t).
$$
It is natural to call $\gamma(T)$ the \emph{holonomy} of the loop $X$ and we denote it $H(X,T)$. 
If $q$ is a metric on $\g$ then we define the \emph{holonomy locus} of $q$, denoted $\mathcal{H}(q)$, as follows
$$
\mathcal{H}(q)=\left\{ H(X,T) \ \mid \ X \ \text{is a non-zero T-periodic solution of} \ V_q \right\}.
$$
\begin{remark}
We have the following
\begin{enumerate}
    \item The holonomy locus $\mathcal{H}(q)$ is closed under positive iteration. 
    Indeed, if $\delta=\gamma(T)$ corresponds to a
    $T$-periodic solution $X$, then
    $$
    \delta^n=\gamma(nT)=H(X,nT)\in\mathcal{H}(q),\qquad n\geq1,
    $$
    since $X$ is also $nT$-periodic.

    \item For every nonzero singularity $w$ of $V_q$,
    \[
    \{\exp(tw)\mid t\in\R^*\}\subset\mathcal{H}(q).
    \]
    Indeed, $w$ and $-w$ define constant solutions, periodic for every
    $T>0$, whose holonomies are $\exp(Tw)$ and $\exp(-Tw)$, respectively.
\end{enumerate}
\end{remark}
\begin{lemma}\label{holonomy-lemma}
    Let $q$ be a left-invariant metric on a connected Lie group $G$ and $\Gamma\subset G$ be a discrete subgroup. 
    Then $\Gamma\backslash G$ admits a closed geodesic if and only if a conjugate of $\Gamma$ intersects $\mathcal{H}(q)$.
\end{lemma}

\begin{proof}
Let $\alpha$ be a closed geodesic in $\Gamma\backslash G$, of period $T>0$, and let $\widetilde\alpha:\R\to G$ be a lift.
Set $g=\widetilde\alpha(0)$. 
There exists $\delta\in\Gamma$ such that
$$
\widetilde\alpha(t+T)=\delta\widetilde\alpha(t)
\qquad\text{for all }t\in\R.
$$
The velocity curve $X_{\widetilde{\alpha}}$ is a nonzero $T$-periodic solution of $V_q$. 
Let $\gamma$ be the developing curve of $X_{\widetilde{\alpha}}$ starting at the identity. Precisely, $\gamma(t)=g^{-1}\widetilde\alpha(t)$. 
Thus
$$
H(X,T)=\gamma(T)=g^{-1}\delta g
\in g^{-1}\Gamma g\cap\mathcal{H}(q).
$$

Conversely, suppose that $H(X,T)=g^{-1}\delta g$ for some $\delta\in\Gamma$ and some nonzero $T$-periodic solution $X$ of $V_q$.
Let $\gamma$ be its developing curve starting at the identity.
The curve $t\mapsto g\gamma(t)$ is a geodesic and satisfies
$$
g\gamma(t+T)=gH(X,T)\gamma(t)=\delta g\gamma(t).
$$
Its projection to $\Gamma\backslash G$ is therefore a closed geodesic of period $T$.
\end{proof}

\section{Interesting Examples}
\subsection{Compact quotients of Lie groups with nontrivial center}
\subsubsection{Compact isometric actions}
\begin{proposition}{(c.f. Lemma~2.1 in \cite{AlloutBelkacemZeghib2024closed})}\label{prop:compact-action}
Let $(M,g)$ be a compact semi-Riemannian manifold and $K$ be a non-trivial connected
compact Lie group acting isometrically on $M$. 
Then $M$ admits closed geodesics.
More precisely, every closed one-parameter group in $K$ admits a non-trivial orbit which is a geodesic.
\end{proposition}

\begin{corollary}\label{cor:center}
Let $G$ be a connected Lie group whose algebra $\g$ has a non-trivial center $\mathfrak z(\g)$.
Then for any left-invariant metric on $G$ and any cocompact lattice $\Gamma\subset G$, the compact quotient
$\Gamma\backslash G$ admits closed geodesics.
\end{corollary}
\begin{proof}
Choose $0\neq z\in\mathfrak z(\g)$. The right action by $\exp(tz)$ descends
to a flow on $\Gamma\backslash G$. This flow is isometric for any left-invariant metric since left and right actions by $\exp(tz)$ on $G$ are the same (i.e. $\exp(tz)$ is central). In particular, $\exp(tz)$ preserves a Riemannian metric on $\Gamma\backslash G$. Thus, Proposition \ref{prop:compact-action} applies.
\end{proof}

\begin{example}
If $G$ is a connected nilpotent Lie group, then it has a non-trivial center
and Corollary~\ref{cor:center} applies. Moreover, for any cocompact lattice $\Gamma$ of a nilpotent Lie group $N$, the intersection of $\Gamma$ with the center $\mathcal{Z}(N)$ of $N$ is a lattice $\Gamma_0$ of $\mathcal{Z}(N)$. 
The quotient $\Gamma\backslash N$ is a principal $\mathcal{Z}(N)/\Gamma_0$-bundle over $(\Gamma/\Gamma_0)\backslash(N/\mathcal{Z}(N))$ with totally geodesic flat fibers.

\end{example}

\subsection{Quotients of the Heisenberg groups}
Let $\Heis_{2n+1}$ denote the Heisenberg group of dimension
$2n+1$. It is characterized by the fact that it is a two-step nilpotent Lie group with
one-dimensional center $Z$ isomorphic to $\R$. Every cocompact lattice $\Gamma$ intersects $Z$ in a lattice,
and its projection to $\Heis_{2n+1}/Z\simeq \R^{2n}$ is a  lattice. 

\begin{proposition}\label{prop:heis-nondegenerate}
Let $q$ be a left-invariant metric on $\Heis_{2n+1}$ such that the center $Z$ is non-degenerate.
Then every compact quotient $\Gamma\backslash\Heis_{2n+1}$ admits both spacelike and timelike geodesics.
\end{proposition}
\begin{proof}
Assume first that the center $Z$ of $\Heis_{2n+1}$ is spacelike. For any lattice $\Gamma\subset\Heis_{2n+1}$, the quotient $\Gamma\backslash\Heis_{2n+1}$ is a circle bundle over the torus $\T^{2n}$ with spacelike geodesic fibers.

Let $\mathfrak z^\perp =W\subset\mathfrak{heis}_{2n+1}$ denote the subspace orthogonal to the center $\mathfrak z$. The Euler--Arnold vector field $V_q$ vanishes on $W$. Indeed, for any $x\in W$ we have $q(\ad_x^*x,y)=q(x,\ad_x y)$ for every $y\in \mathfrak{heis}_{2n+1}$. But $\ad_x y\in \mathfrak z$ since $[\mathfrak{heis}_{2n+1},\mathfrak{heis}_{2n+1}]=\mathfrak{z}$. This shows that $q(\ad_x^*x,y)=0$ for all $y$, that is, $\ad_x^*x=0$.

So, $\exp(tx)$ is a geodesic for every $x\in W$ and projects to a line $l_x$ through the origin in $\Heis_{2n+1}/Z\simeq \R^{2n}$. Choose $x\in W$ a timelike vector such that $l_x$ is a rational line, i.e $l_x$ intersects $\Gamma/\Gamma_0$ non-trivially (such $x$ exists since timelike vectors constitute an open cone). This means that there is $t_0>0$ such that $z_0.\exp(t_0x)\in \Gamma$ for some $z_0\in Z$.

Recall that $g\exp(t_0x)g^{-1}=[g,\exp(t_0x)]\exp(t_0x)$ where $[g,\exp(t_0x)]$ denotes the commutator of $g$ and $\exp(t_0x)$. We choose $g$ such that $[g,\exp(t_0x)]=z_0$. This shows that $g\exp(t_0x)$ projects to a closed timelike geodesic in $\Gamma\backslash\Heis_{2n+1}$.

We treat the case where the center is timelike similarly.
\end{proof}

\subsubsection{A remarkable flat Lorentzian example}
\begin{proposition}\label{prop:heis-flat}
Let $q$ be a left-invariant Lorentzian metric on $\Heis_3$ such that the center $Z$ is
lightlike. Then $q$ is flat, and there exists a lattice $\Gamma\subset \Heis_3$ such that the fibers of the principal circle bundle $\Gamma\backslash\Heis_3$ are the only closed geodesics.
\end{proposition}
\begin{proof}
Since the center is lightlike, we can choose a basis $\{X,Y,Z\}$
of $\mathfrak{heis}_3$ such that $[X,Y]=Z$ and
\[
q(xX+yY+zZ)=y^2+2xz.
\]
This metric is flat (c.f. Table 2. in \cite{ChaibFerreiraZeghib2026}) and the associated Euler--Arnold vector field is
\[
V_q(x,y,z)=(0,x^2,-xy).
\]
Its singular set is the abelian ideal $P=\operatorname{span}\{Y,Z\}$.
Outside $P$, the coordinate $x$ is constant and $y$ is strictly increasing. 
Thus the only periodic solutions are the constant ones, and
\[
\mathcal{H}(q)=A\setminus\{e\},
\qquad A:=\exp(P).
\]
Notice that $A$ is a normal subgroup containing the center $C:=\exp(\R Z)$.

Let $\pi:\Heis_3\to\Heis_3/C\simeq\R^2$ be the natural projection,
where the quotient is identified with $\R^2$ using the images
of $X$ and $Y$. Fix $\theta\in\R\setminus\mathbb{Q}$ and set
\[
\Gamma=
\big\langle \exp(X),\,\exp(\theta X+Y),\,\exp(Z)\big\rangle.
\]
This is a cocompact lattice: it is the image of the standard Heisenberg lattice $\Heis_3(\Z)$ under the automorphism induced by
\[
X\longmapsto X,\qquad Y\longmapsto\theta X+Y,\qquad Z\longmapsto Z.
\]
Moreover,
\[
\pi(\Gamma)=\mathbb{Z}(1,0)+\mathbb{Z}(\theta,1),
\qquad
\pi(A)=\{0\}\times\R.
\]
Since $\theta$ is irrational, these two subgroups intersect only
at the origin. Therefore
\[
\Gamma\cap A=\Gamma\cap C.
\]
As $A$ is normal and $C$ is central, for every $g\in\Heis_3$ we obtain
\[
g^{-1}\Gamma g\cap\mathcal{H}(q)
=
g^{-1}(\Gamma\cap A)g\setminus\{e\}
=
(\Gamma\cap C)\setminus\{e\}.
\]
This shows that the fibers are the only closed geodesics.
\end{proof}

\subsection{A complete closed Lorentz 3-manifold without non-spacelike closed geodesics}
Galloway's example in \cite{galloway1986compact} is incomplete and has no non-spacelike closed geodesics.
Theorem \ref{thm:sol-intro} gives a complete example with the same property.

\begin{proof}[Proof of Theorem \ref{thm:sol-intro}]
Use the metric $q_2$ and notation of Subsec.~\ref{subsec:sol-metric}, with
$$
q_2(xh+ye_1+ze_2)=x^2-y^2+2yz.
$$
Its Euler--Arnold equation is 
$$
\dot x=y^2,\qquad \dot y=-xy,\qquad \dot z=x(z-2y).
$$
The metric $q_2$ is complete (see Subsec. 7.2 in \cite{CF-partII}).
If a solution is periodic, integration of $\dot x=y^2$ forces $y=0$.
Then $x$ is constant and $z(t)=z(0)e^{x(t)}$, so periodicity forces $xz(0)=0$.
The only periodic solutions are consequently the singularities in $\R h\cup\R e_2$.
The first line is spacelike and the second lightlike. The holonomy locus $\mathcal{H}(q)$ is therefore $(\exp(\R h)\cup \exp(\R e_2))\setminus \{e\}$.

Let $\Gamma\subset \Sol$ be a cocompact lattice. Its intersection with the nilradical $\R^2$ is a lattice that intersects the normal eigen-directions $\exp(\R e_1)$ and $\exp(\R e_2)$ trivially. 
This shows that all conjugates of $\Gamma$ intersect the normal direction trivially.
Hence, the conjugates of $\Gamma$ intersect the holonomy locus only in $\exp(\R h)\setminus \{e\}$. This finishes the proof.
\end{proof}

\section{A complete compact manifold of signature \texorpdfstring{$(2,4)$}{(2,4)} without closed geodesics}

\subsection{A metric on \texorpdfstring{$\Euc$}{Euc}}
The group $\Euc=\rm{SO}(2)\ltimes\R^2$ denotes the identity component of the
isometry group of the Euclidean plane. Its Lie algebra $\mathfrak{euc}$, endowed with the basis $e,f_2,f_3$, with
bracket relations $[e,f_2]=-f_3$ and $[e,f_3]=f_2$. Let $q_1$ be the Lorentzian metric on $\mathfrak{euc}$ given by
$$
q_1(x,y,z)=2xy+z^2.
$$
The geodesic flow on $\mathfrak{euc}$ is given by the system of equations
$$
\mathcal E_1:\quad
\begin{cases}
\dot x=-xz,\\
\dot y=z(y-x),\\
\dot z=x^2.
\end{cases}
$$
The abelian subalgebra $Q:=\{x=0\}$ is invariant and the geodesic vector field on $Q$ has the form $(0,y,z)\mapsto(0,yz,0)$.
We get that there are no periodic solutions except the constant ones.
Every Lorentz metric on $\mathfrak{euc}$ is geodesically complete, by Guediri \cite{Guediri1996}; see also \cite{ElshafeiFerreiraSanchezZeghib2024}. In particular, $q_1$ is.
\begin{remark}
    The left-invariant metric used here is the same as in Theorem 1.1 in \cite{AlloutBelkacemZeghib2024closed}.
\end{remark}

\subsection{A metric on \texorpdfstring{$\Sol$}{Sol}}\label{subsec:sol-metric}
Let $\Sol=\R\ltimes\R^2$ be the solvable three-dimensional group given by the
representation $t\mapsto\left(\begin{smallmatrix}e^t&0\\0&e^{-t}\end{smallmatrix}\right)$.
Its Lie algebra, denoted by $\mathfrak{sol}$, has a basis $h,e_1,e_2$ with bracket
relations
$$
[h,e_1]=e_1\quad\text{and}\quad[h,e_2]=-e_2.
$$
With respect to this basis, let $q_2$ be the Lorentzian metric on $\mathfrak{sol}$ given by
$$
q_2(x,y,z)=x^2-y^2+2yz.
$$
The geodesic flow on $\mathfrak{sol}$ is given by the following system of equations
$$
\mathcal E_2:\quad
\begin{cases}
\dot x=y^2,\\
\dot y=-xy,\\
\dot z=x(z-2y).
\end{cases}
$$
This system has only two singular directions: $\R h$ and $\R e_2$.
One can check that there are no periodic solutions except the constant ones.
This metric is geodesically complete; see, e.g., \cite[Subsection 7.2]{CF-partII}.

\subsection*{The main example}
Put $G=\Sol\times\Euc$ and identify its Lie algebra with $\g=\mathfrak{euc}\oplus\mathfrak{sol}$, endowed
with the metric $q=q_1\oplus q_2$ of signature $(2,4)$. Both factors are geodesically complete, so the
product metric is geodesically complete.

The singularities of the geodesic vector field for $q$ are exactly the elements of the four planes
(abelian subalgebras)
$$
\begin{aligned}
\mathfrak p_1&=\Span(f_2,h),&\quad \mathfrak p_2&=\Span(f_2,e_2),&
\quad\mathfrak p_3&=\Span(f_3,h),&\quad\mathfrak p_4&=\Span(f_3,e_2).
\end{aligned}
$$
Put $P_i:=\exp(\mathfrak p_i)$ for $1\leq i\leq4$. Each $P_i$ is isomorphic to $\R^2$.
Since the geodesic equation splits into the two factor equations, its only periodic solutions on $\g$ are the constant ones. In other words, $\mathcal{H}(q)=\bigcup P_i\setminus \{e\}$.

We recall the following Lemma.
\begin{lemma}\label{lem:product-closing}
Let $\Gamma\subset G$ be a discrete subgroup and endow $M:=\Gamma\backslash G$ with the metric induced by
the left-invariant metric on $G$ generated by $q$. Then $M$ admits a nonconstant closed geodesic if and
only if there exist a non-trivial element $\gamma\in\Gamma$ and $g\in G$ such that $\gamma\in gP_ig^{-1}$ for some $1\leq i\leq4$.
\end{lemma}
\begin{proof}
This follows immediately from Lemma~\ref{holonomy-lemma}.
\end{proof}

\subsection{Irreducible lattices}
We think of $G=\Sol\times\Euc$ as the semidirect product
$(\operatorname{SO}(2)\times\R)\ltimes(V_1\oplus V_2)$, where $V_1$ and $V_2$ are the invariant planes
$\exp(\Span(f_2,f_3))$ and $\exp(\Span(e_1,e_2))$, respectively.
If an element $a_0\in \operatorname{SO}(2)\times\R$, with a non-vanishing $\R$-factor, acts by a linear map $A$
preserving a lattice $\Lambda\subset\R^4=V_1\oplus V_2$, then
the semidirect product $\Z\ltimes_A\Lambda$ is a co-compact lattice in $G$.

The condition that $A$ preserves a lattice in $\R^4$ means that $A$ is conjugate to an element of
$\operatorname{SL}(4,\Z)$. In other words, we are looking for a rotation $B_1\in \operatorname{SO}(2)$ and a diagonal hyperbolic
element $B_2=\diag(\lambda,\lambda^{-1})$, with $\lambda>1$, such that the $4\times4$ matrix $A=\diag(B_1,B_2)$ is $\rm GL(4,\R)$-conjugate to an element of $\operatorname{SL}(4,\Z)$.

Let $A$ be such an element with $B_1$ an irrational rotation, and denote by $\Lambda\subset V_1\oplus V_2$ its invariant
lattice. We call $\Gamma=\Z\ltimes_A\Lambda$ an \emph{irreducible lattice} of $G$.
\begin{remark}
Irreducible lattices in $G$ are characterized by the property that their intersections
with both factors $\Sol$ and $\Euc$ are trivial. This justifies the terminology.
\end{remark}

\subsubsection{Construction of irreducible lattices}[c.f. Remark 5.2 in \cite{AlloutBelkacemZeghib2024closed}]
Let $\Q(\sqrt2)=\{a+b\sqrt2\mid a,b\in\Q\}$ denote the quadratic
extension of $\Q$ by $\sqrt2$. The ring of integers of $\Q(\sqrt2)$ is
$$
\Z[\sqrt2]=\{m+n\sqrt2\mid m,n\in\Z\}.
$$
The Galois automorphism $\sigma:\Q(\sqrt2)\to\Q(\sqrt2)$, given by $\sigma(a+b\sqrt2)=a-b\sqrt2$, leaves invariant
the ring $\Z[\sqrt2]$. The subset
$$
\{(x,\sigma(x))\mid x\in\Z[\sqrt2]\}
$$
is a lattice in $\R^2$. In particular, the subset
$$
\Lambda_0=\{(x,y,\sigma(x),\sigma(y))\mid x,y\in\Z[\sqrt2]\}
$$
is a lattice in $\R^4$. Let $\operatorname{SL}(2,\Z[\sqrt2])$ denote the group of $2\times2$ matrices of determinant one and
entries in the ring $\Z[\sqrt2]$. The group $\operatorname{SL}(2,\Z[\sqrt2])$ is clearly dense in $\operatorname{SL}(2,\R)$.

We consider the embedding
$$
\rho:\operatorname{SL}(2,\Z[\sqrt2])\longrightarrow \operatorname{SL}(2,\R)\times \operatorname{SL}(2,\R),\qquad
B\longmapsto\diag(B,\sigma(B)).
$$
It follows that $\rho(\operatorname{SL}(2,\Z[\sqrt2]))\subset \operatorname{SL}(4,\R)$ is discrete and preserves the lattice $\Lambda_0$
(actually $\rho(\operatorname{SL}(2,\Z[\sqrt2]))$ is a lattice in $\operatorname{SL}(2,\R)\times \operatorname{SL}(2,\R)$, i.e of finite co-volume, see e.g. \cite[Sect. 2.9.C2]{GromovPansu}).

Now take any $B$ such that $0<\tr(B)<2$ and $2<\tr(\sigma(B))=\sigma(\tr(B))$. Then, $A=\diag(B_1,B_2)$ preserves $\Lambda_0$. The rotation $B$ is irrational. Hence, $\Gamma=\Z\ltimes_A\Lambda$ is an irreducible lattice of $G$.

\subsubsection{Properties of \texorpdfstring{$\Gamma$}{Gamma}}
\begin{lemma}\label{lem:W}
Let $\Gamma=\Z\ltimes_A\Lambda$ be an irreducible lattice as above. 
Then $\Lambda$ intersects $P_2+P_4\subset V_1\oplus V_2$ trivially.
\end{lemma}
\begin{proof}
    Let $\gamma\in (P_2+ P_4)\cap \Lambda$ and recall that $P_2+ P_4=\exp(\operatorname{span}(f_2,f_3,e_2))$. 
    Since the action of $A$ on the direction $\exp(\R e_2)$ is contracting (up to forward or backward iteration) and its action on $\R f_2\oplus \R f_3$ is by an irrational rotation, we deduce that the orbit of $\gamma$ is bounded and therefore finite. 
    Hence $\gamma$ must be zero.
\end{proof}
\begin{corollary}\label{cor:all-conjugates}
All conjugates of $P_i$, for $1\leq i\leq4$, intersect $\Gamma$ trivially.
\end{corollary}
\begin{proof}
Lemma \ref{lem:W} shows this for $P_2$ and $P_4$. 
Suppose that $\gamma\in P_3=\exp(\operatorname{span}(f_3,h))$ is conjugate to an element of $\Gamma$.
That is, $\gamma=\exp(tf_3+sh)=\exp(tf_3)\exp(sh)$ for some real numbers $t,s\in\R$. 
It follows that the adjoint action of $\gamma$ on $V_1$ is trivial. 
This forces $\gamma$ to be in $\Lambda$ since the adjoint action of any element of $\Gamma$ on $V_1$ is either trivial or via an irrational rotation. 
This shows that $s=0$, which means that $\gamma=\exp(tf_3)$, which is impossible by Lemma \ref{lem:W}.
The same holds for $P_1$. 
\end{proof}

\subsubsection{A compact quotient}
Combining Lemma~\ref{lem:product-closing} with Corollary~\ref{cor:all-conjugates}, we obtain Theorem~\ref{thm:main},
more precisely, 
\begin{corollary}
Endow the compact quotient $M=\Gamma\backslash G$ with the metric induced by the left-invariant
metric on $G$ generated by $q$. Then $M$ is a geodesically complete manifold of signature $(2,4)$ without
closed geodesics.
\end{corollary}
\begin{proof}
By Corollary~\ref{cor:all-conjugates},
no conjugate of any $P_i$ contains a non-trivial element of $\Gamma$. Lemma~\ref{lem:product-closing} finishes the proof.
\end{proof}

\section*{AI declaration}
AI tools were used only for writing assistance. All mathematical content was developed and verified by the authors.

\bigskip

\section*{Acknowledgments}
The second named author acknowledges the support of CMAT (Centro de Matemática da Universidade do Minho). His research was financed by Portuguese Funds through FCT (Fundação para a Ciência e a Tecnologia, I.P.) within the Project UID/00013/2025 and also within the doctoral grant UI/BD/154255/2022.
\printbibliography

\end{document}